\documentclass[11pt]{amsart}

\usepackage{amsmath, amssymb, amsfonts}
\usepackage{graphicx, overpic}
\usepackage{fullpage}
\usepackage{color}
\usepackage{enumerate}

\usepackage{multicol}

\usepackage[margin=1.05in]{geometry}

\theoremstyle{plain}
\newtheorem{thm}{Theorem}[section]
\newtheorem{lemma}[thm]{Lemma}
\newtheorem{prop}[thm]{Proposition}

\numberwithin{equation}{section}
\numberwithin{figure}{section}

\theoremstyle{definition}
\newtheorem{definition}[thm]{Definition}

\newtheorem{question}[thm]{Question}
\newtheorem{construction}[thm]{Construction}

\newcommand{\G}{\Gamma}
\newcommand{\Z}{\mathbb{Z}}
\newcommand{\R}{\mathbb{R}}

\newcommand{\cP}{\mathcal{P}}

\newcommand{\cD}{\mathcal{D}}
\newcommand{\cX}{\mathcal{X}}
\newcommand{\cY}{\mathcal{Y}}

\newcommand{\cW}{\mathcal{W}}
\newcommand{\cZ}{\mathcal{Z}}

\newcommand{\hY}{\widehat{\mathcal{Y}}}
\newcommand{\hZ}{\widehat{\mathcal{Z}}}

\def\polhk#1{\setbox0=\hbox{#1}{\ooalign{\hidewidth
    \lower1.0ex\hbox{$\,\lhook$}\hidewidth\crcr\unhbox0}}}

\title{Topological rigidity fails for quotients of the Davis complex}
\author{Emily Stark}
\date{\today}
\thanks{2010 {\it Mathematics Subject Classification.} Primary 51F15; Secondary 20E07; 20F55; 20F67}
\begin{document}
 
\maketitle

\begin{abstract} 
A Coxeter group acts properly and cocompactly by isometries on the Davis complex for the group; we call the quotient of the Davis complex under this action the {\it Davis orbicomplex} for the group. We prove the set of finite covers of the Davis orbicomplexes for the set of one-ended Coxeter groups is not topologically rigid. We exhibit a quotient of a Davis complex by a one-ended right-angled Coxeter group which has two finite covers that are homotopy equivalent but not homeomorphic. We discuss consequences for the abstract commensurability classification of Coxeter groups.

\end{abstract}

\section{Introduction}

The notion of topological rigidity has its roots in the setting of manifolds. A closed manifold $M$ is called {\it topologically rigid} if every homotopy equivalence from $M$ to another closed manifold is homotopic to a homeomorphism. 
A well-known example of this phenomenon is the {\it Poincar\'{e} Conjecture}, which states that the $3$-sphere is topologically rigid and was proven by Perelman. Many lens spaces are examples of $3$-manifolds that are not topologically rigid.  The {\it Borel conjecture} states that closed aspherical manifolds are topologically rigid. The conjecture was proven for manifolds of dimension at least five whose fundamental group is either Gromov hyperbolic or CAT$(0)$ by Bartels and L\"{u}ck \cite{bartelslueck}, building on the techniques of Farrell and Jones \cite{farrelljones}.

The definition of topological rigidity extends from manifolds to orbifolds and to classes of topological spaces. Background on orbifolds and orbifold homeomorphisms is given by Kapovich \cite[Chapter 6]{kapovich} and Ratcliffe \cite[Chapter 13]{ratcliffe}. An {\it orbicomplex} is the union of orbifolds identified along homeomorphic suborbifolds, and the notion of homeomorphism extends to these spaces as well. 

\begin{definition}
 Let $\cX$ be a class of topological spaces, orbifolds, or orbicomplexes. The class $\cX$ is said to be {\it topologically rigid} if for all $X_1, X_2 \in \cX$, if $\pi_1(X_1) \cong \pi_1(X_2)$, then $X_1$ and $X_2$ are homeomorphic.
\end{definition}

Simplicial graphs provide a simple example of a class of spaces that is not topologically rigid. 
More generally, for graphs of spaces with one-ended universal covers the presence of topological rigidity is more subtle. Lafont proved that {\it simple, thick, $n$-dimensional hyperbolic piecewise-manifolds} are topologically rigid for $n \geq 2$ \cite{lafont-2}\cite{lafont-3}\cite{lafont}. 
In dimension two, these spaces decompose as graphs of spaces with vertex spaces that are compact hyperbolic surfaces with boundary, edge spaces that are circles, and edge-to-vertex space inclusions that identify the boundary components of the surfaces so that each boundary component is identified to at least two others; the higher-dimensional analogues are similar. 
The orbicomplexes considered in this paper also have hyperbolic  fundamental groups and codimension-1 singularities along embedded locally geodesic $1$-complexes; we show that finite covers of these spaces do not exhibit topological rigidity.  

The spaces studied in this paper have fundamental groups of the following form. If $\G$ is a finite simplicial graph with vertex set $V\G$ and edge set $E\G$, the {\it right-angled Coxeter group $W_{\G}$ with defining graph $\Gamma$} has generating set $V\G$ and relations $s^2=1$ for all $s \in V\G$ and $st = ts$ whenever $\{s,t\} \in E\G$. If $W_{\G}$ is a right-angled Coxeter group, then $W_{\G}$ acts properly and cocompactly by isometries on its {\it Davis complex} $\Sigma_{\G}$. The quotient of this space under the action of the right-angled Coxeter group $W_{\G}$ is called the {\it Davis orbicomplex} $\cD_{\G} := W_{\G} \backslash\backslash \Sigma_{\G}$. Background on Coxeter groups and the Davis complex is given by Davis \cite{davis-book}.
A description of reflection orbicomplexes related to those described here can be found in \cite[Section 5.2]{stark} and \cite[Section 3]{danistarkthomas}.

A natural setting for questions of topological rigidity for spaces with fundamental groups right-angled Coxeter groups and their finite-index subgroups is the set of Davis orbicomplexes and their finite-sheeted covers, as these spaces have a natural orbicomplex structure. If the graph~$\G$ has no edges, so that $W_{\G}$ is a free product of groups isomorphic to $\Z/2\Z$, then the set of Davis orbicomplexes and their finite-sheeted covers is not topologically rigid: such a group is the fundamental group of an orbicomplex which is finitely covered by a finite simplicial graph. So, one may ask if topological rigidity holds if one restricts to one-ended right-angled Coxeter groups. The main result of this paper is the following.

\begin{thm} \label{toprigidity}
 Let $\mathcal{X}$ be the set of finite covers of the Davis orbicomplexes for the set of one-ended right-angled Coxeter groups. The set $\mathcal{X}$ is not topologically rigid. 
\end{thm}

To prove Theorem \ref{toprigidity}, we construct an example of a one-ended right-angled Coxeter group $W_{\G}$  so that if $\cD$ is its Davis orbicomplex, then there exist two finite covers of $\cD$ that have the same fundamental group but are not homeomorphic.
The orbicomplex $\cD$ contains a singular subspace that is finitely covered by a graph; we find two non-homeomorphic covers of the graph that extend to covers of $\cD$ so that the covers have the same fundamental group. 
The construction of $W_{\G}$ and $\cD$ is given in Section~\ref{sec:Davis}, and the finite covers are described in Section~\ref{sec:covers}. 

In Proposition~\ref{prop:tf_covers}, we employ similar ideas to produce further finite covers of $\cD$ that are quotients of the Davis complex by isomorphic torsion-free subgroups of $W_{\G}$ and so that these covers are not homeomorphic.

As shown by Crisp--Paoluzzi \cite{crisp-paoluzzi} using the work of Lafont \cite{lafont}, there are one-ended right-angled Coxeter groups which are not virtually manifold groups for which the Davis orbicomplex together with its finite-sheeted covers is topologically rigid. 
So, we state the following question. 

\begin{question}
 For which set $\cW$ of Coxeter groups is the set of Davis orbicomplexes for groups in $\cW$ together with their finite-sheeted covers topologically rigid? 
\end{question}

Relatedly, Xie \cite{xie} proved the set of quotients of  {\it Fuchsian buildings} by the action of a cocompact lattice is topologically rigid. An interesting problem is to determine the set of lattices in the isometry group of a hyperbolic building which have a set of quotients that is topologically rigid. 

\begin{definition}
 A class of spaces $\mathcal{X}$ is said to be {\it closed under finite covers} if whenever $X \in \mathcal{X}$, all finite-sheeted covering spaces of $X$ are in $\mathcal{X}$. 
\end{definition}

Topological rigidity of $\mathcal{X}$, a class of spaces closed under finite covers, has applications in the study of the abstract commensurability classes of the fundamental groups of spaces in $\mathcal{X}$. 
Recall that two groups are {\it abstractly commensurable} if they contain finite-index subgroups that are isomorphic.
If $\cX$ is a topologically rigid class of spaces closed under finite covers and 
$X_1, X_2 \in \mathcal{X}$, then $\pi_1(X_1)$ and $\pi_1(X_2)$ are abstractly commensurable if and only if $X_1$ and $X_2$ have homeomorphic finite-sheeted covering spaces. Thus, in this setting, topological invariants can be used to distinguish abstract commensurability classes. 
For example, this technique was employed by Crisp--Paoluzzi \cite{crisp-paoluzzi} and Dani-Stark-Thomas \cite{danistarkthomas} for certain right-angled Coxeter groups and by the author for related surface group amalgams \cite{stark}. 
A survey on the use of orbifolds in the study of commensurability classes is given by Walsh \cite{walsh}, and background on commensurability classification is given by Paoluzzi \cite{paoluzzi}. 

This paper was motivated by an interest in understanding the abstract commensurability classes of Coxeter groups. Dani--Thomas \cite{danithomas} provide a quasi-isometry classification within a class of hyperbolic one-ended right-angled Coxeter groups, and in joint work with Dani and Thomas \cite{danistarkthomas}, we refine their work to give an abstract commensurability classification for a subclass of these groups. 
In particular, it was of interest to determine whether topological rigidity holds for finite covers of the Davis orbicomplex, which are natural spaces for right-angled Coxeter groups and their finite-index subgroups. 

\subsection*{Acknowledgments} The author thanks Pallavi Dani and Anne Thomas for enlightening conversations during our work on \cite{danistarkthomas} and for comments on a draft of this paper. The author thanks the anonymous referee for helpful comments. The author was partially supported by the Azrieli Foundation. 
 
\section{The Davis orbicomplex} \label{sec:Davis}

    \begin{figure}[h]
     \begin{overpic}[scale=.3,  tics=5]{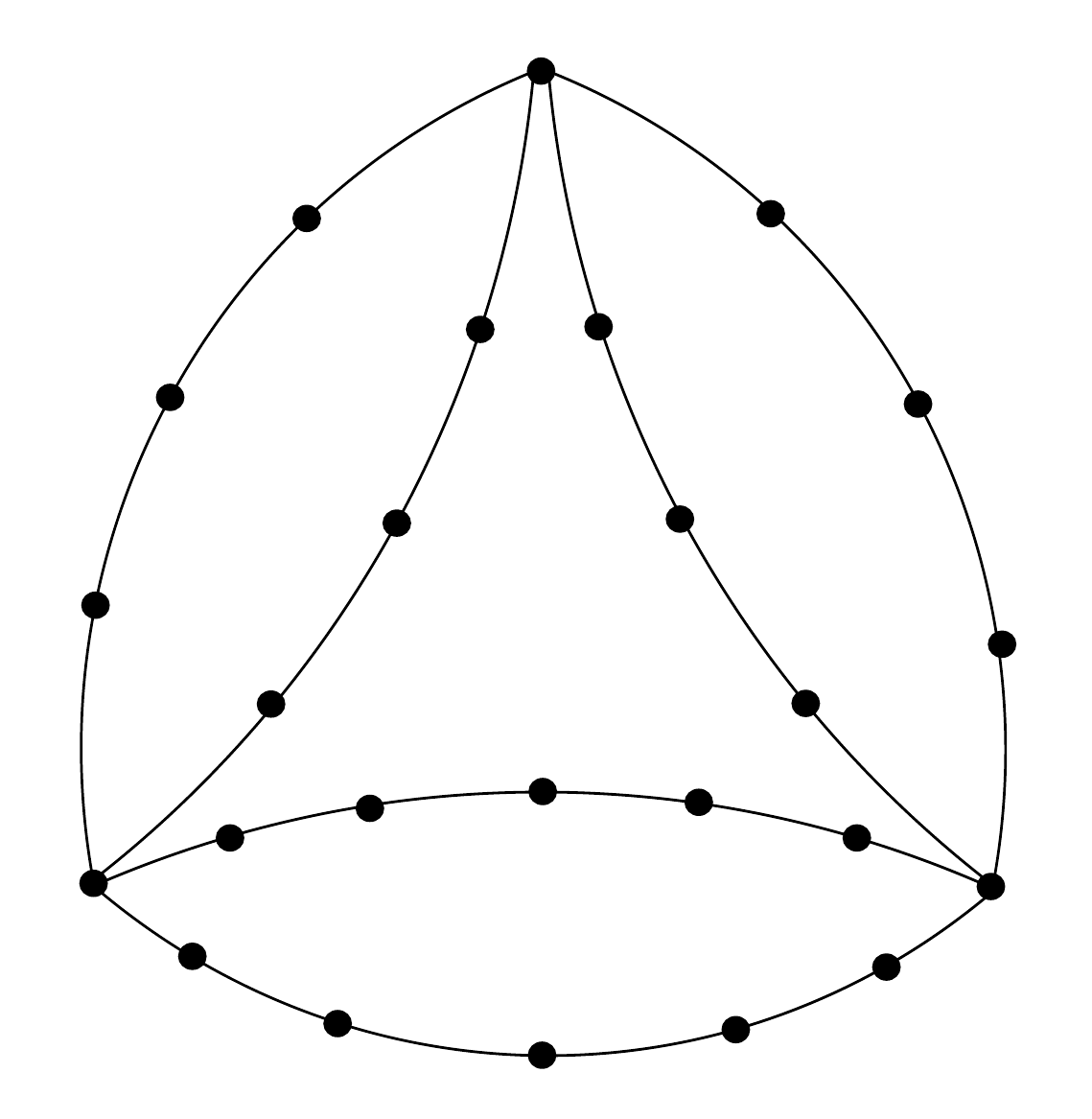}
    \put(44,99){$v_3$}
    \put(-4,15){$v_1$}
    \put(92,15){$v_2$}
     \end{overpic}
     \caption{{\small The graph $\Gamma$ defining the group $W_{\G}$ }}
     \label{figure:graph}
    \end{figure}

     \begin{definition}[The group $W$] \label{def:W}
 Let $W = W_{\G}$ be the right-angled Coxeter group with defining graph $\G$ given in Figure \ref{figure:graph}. 
\end{definition}

\begin{construction}[The Davis orbicomplex $\cD$ for $W$] \label{con:D}
 Let $W$ be the right-angled Coxeter group given in Definition \ref{def:W}. The Davis orbicomplex $\cD$ for $W$ has the following form, which is illustrated in Figure \ref{figure:orbicomplex}. The space $\cD$ is an orbicomplex whose underlying space is topologically the cone on the defining graph $\G$. The space $\cD$ may be viewed as a graph of spaces with vertex spaces $2$-dimensional right-angled reflection orbifolds with boundary, and these orbifolds are identified along their boundary components as follows.

    \begin{figure}
     \begin{overpic}[scale=.73, tics=5]{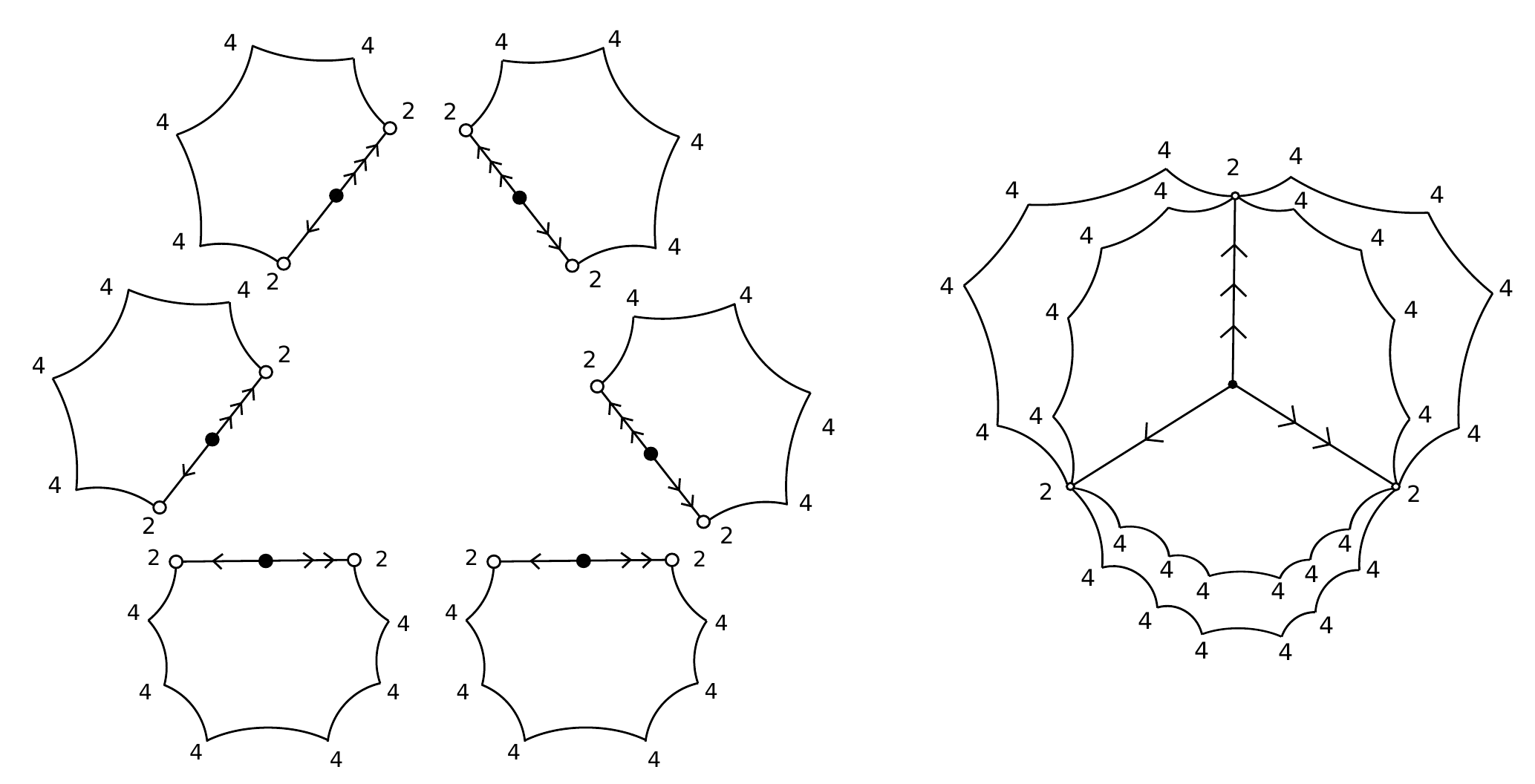}
    \put(70,45){$\cD$}
    \put(0,45){$\cP$}
     \end{overpic}
     \caption{{\small Pictured on the above right is the Davis orbicomplex $\cD$ for the right-angled Coxeter group with defining graph $\G$ given in Definition \ref{def:W}. On the left is illustrated the collection $\cP$ of six right-angled reflection orbifolds that are glued to each other by local isometries to form the Davis orbicomplex. Each edge of these orbifolds is a reflection edge except for the two edges which are glued to other orbifolds as indicated by the arrows. The numbers indicate the order of the isotropy group at the orbifold point. }}
     \label{figure:orbicomplex}
    \end{figure}

 Let $\mathcal{P}$ be the following collection of orbifolds, which will be the vertex spaces of $\cD$. Define a {\it branch} of the graph $\G$ to be an embedded path connecting two vertices of valence four. For a branch $\beta$, let $n_{\beta}$ be the number of vertices of the branch including the endpoints. So, $\G$ has six branches, and if $\beta$ is a branch of $\G$, then $n_{\beta} \in \{5, 7\}$. The right-angled Coxeter group with defining graph a branch $\beta$ is the orbifold fundamental group of the following orbifold $P_{\beta}$. The orbifold $P_{\beta}$ has underlying space a right-angled hyperbolic $(n_{\beta}+1)$-gon, $n_{\beta}$ reflection edges, and one non-reflection edge of length $L>0$. Let $\cP = \{ P_{\beta} \, | \, \beta$ is a branch of $\G\}$; the six orbifolds in $\cP$ are illustrated on the left of Figure \ref{figure:orbicomplex}. 
 
 Identify the orbifolds in $\cP$ along convex suborbifolds to form the orbicomplex $\cD$ as follows. For each orbifold in $\cP$, attach a $0$-cell at the midpoint of each non-reflection edge, creating two non-reflection edges of length $\frac{L}{2}$. Label these non-reflection edges as follows. First, label the three vertices of $\Gamma$ of valence four $\{v_1, v_2, v_3\}$ as illustrated in Figure \ref{figure:graph}. Then, label a non-reflection edge $e_i$ if the edge is incident to the reflection edge corresponding to the vertex $v_i$; this labeling is indicated using arrows in Figure \ref{figure:orbicomplex}. To build the Davis orbicomplex $\cD$, identify the middle vertex of the two non-reflection edges in each polygon in $\mathcal{P}$ and all non-reflection edges of the same label as shown on the right of Figure \ref{figure:orbicomplex}.
 
 It remains to check that the orbifold fundamental group of $\cD$ is the right-angled Coxeter group $W_{\G}$. Indeed, gluing non-reflection edges with the same label to form a single edge creates a single reflection wall perpendicular to this edge; this gluing corresponds to identifying the endpoints of the branches of $\G$ to form $\G$. The claim then follows from arguments similar to those found in \cite[Section 3]{danistarkthomas}. 
\end{construction}

\section{The set of finite-sheeted covers of $\cD$ is not topologically rigid.} \label{sec:covers}

The covers of $\cD$ restricted to the reflection orbifolds in $\cP$ are given by the following two maps, which are illustrated in Figure \ref{figure:2fold}. 

    \begin{figure}
     \begin{overpic}[scale=.4, tics=5]{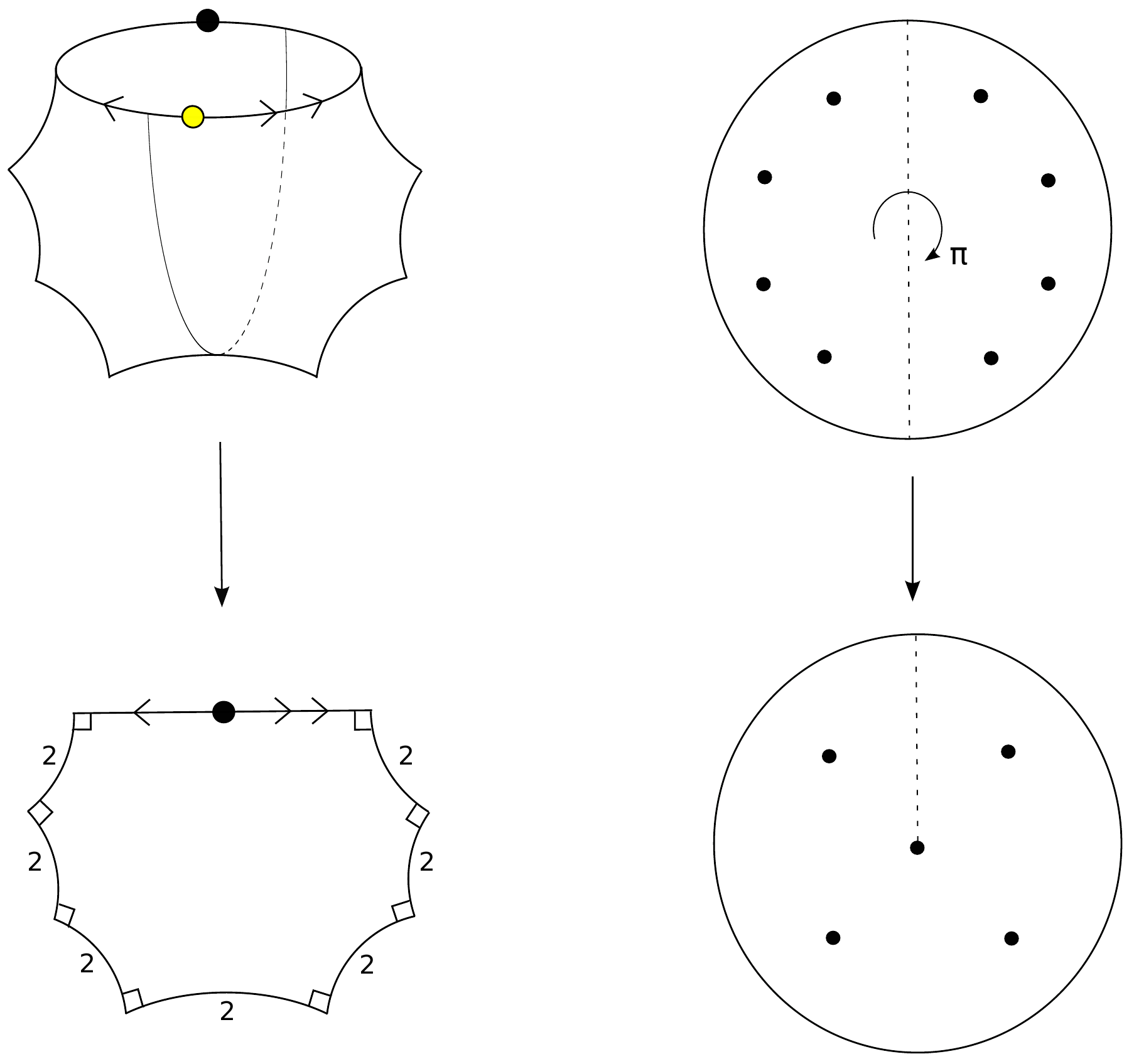}
     \put(22,46){$2$}
     \put(82,46){$2$}
     \put(-3,77){\Small{$2$}}
     \put(-1,67){\Small{$2$}}
     \put(5,59){\Small{$2$}}
     \put(30,59){\Small{$2$}}
     \put(38,67){\Small{$2$}}
     \put(39,77){\Small{$2$}}
     \put(4,30){\Small{$2$}}
     \put(34,30){\Small{$2$}}
     \put(-1,21){\Small{$4$}}
     \put(39,21){\Small{$4$}}
     \put(1,11){\Small{$4$}}
     \put(38,11){\Small{$4$}}
     \put(8,3){\Small{$4$}}
     \put(30,3){\Small{$4$}}
     \put(73,86.5){\Small{$2$}}
     \put(86,86.5){\Small{$2$}}
     \put(66,79){\Small{$2$}}
     \put(91,79){\Small{$2$}}
     \put(66,70){\Small{$2$}}
     \put(91,70){\Small{$2$}}
     \put(72,64){\Small{$2$}}
     \put(86,64){\Small{$2$}}
     \put(72,29){\Small{$2$}}
     \put(72,12.5){\Small{$2$}}
     \put(88,29){\Small{$2$}}
     \put(88,12.5){\Small{$2$}}
     \put(80,14){\Small{$2$}}
     \end{overpic}
     \caption{{\small  Illustrated above are two orbifold covers. On the left, the group $\Z/2\Z$ acts by reflection, identifying the yellow and black vertices, and with quotient space a reflection orbifold in $\cP$. On the right, $\Z/2\Z$ acts by rotation by $\pi$. }}
     \label{figure:2fold}
    \end{figure}
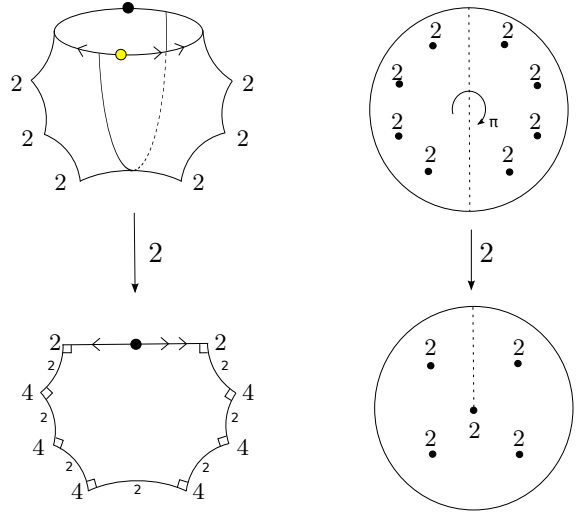

\begin{lemma} \label{lemma:2fold}
 Let $D^2(\,\underbrace{2, \ldots, 2}_\text{$n$}\,)$ denote the orbifold with underlying space a disk and with ramification locus $n$ cone points of order $2$. 
 \begin{enumerate}
  \item[(a)] Each orbifold in $\mathcal{P}$ with $n$ reflection edges is covered by $D^2(\,\underbrace{2, \ldots, 2}_\text{$n-1$}\,)$.
  \item[(b)] The orbifold $D^2(\,\underbrace{2, \ldots, 2}_\text{$2m$}\,)$ double covers $D^2(\,\underbrace{2, \ldots, 2}_\text{$m+1$}\,)$.
 \end{enumerate}
\end{lemma}

\begin{proof}
 The first covering map is realized by reflection: arrange the cone points along a diameter of the disk and reflect across this segment. The second covering map is realized by rotation: arrange the $2m$ cone points symmetrically about a central (non-orbifold) point in the disk and rotate by~$\pi$. 
\end{proof}

    \begin{figure}
     \begin{overpic}[scale=.55, tics=5]{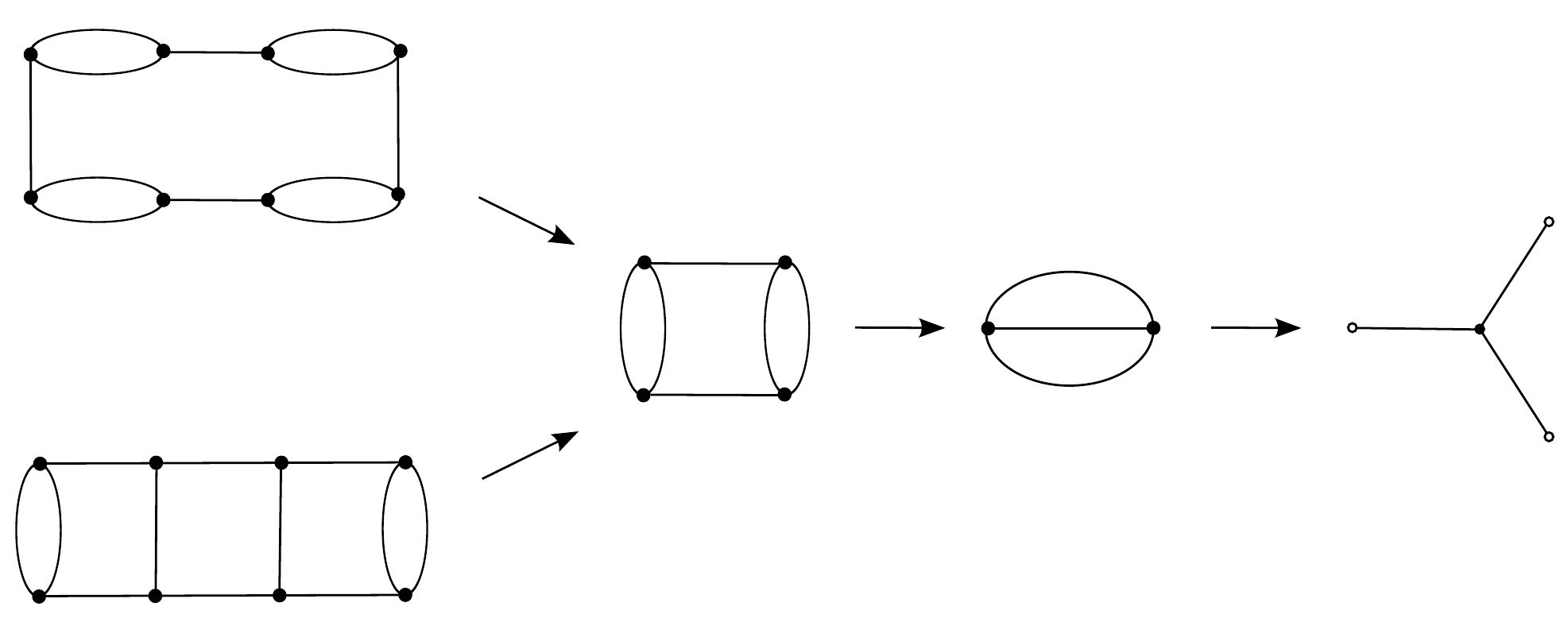}
     \put(33,6){\small{2}}
     \put(33,27){\small{2}}
     \put(57,20){\small{2}}
     \put(79,20){\small{2}}
     \put(86,20){\small{2}}
     \put(100,11){\small{2}}
     \put(100,26){\small{2}}
    \end{overpic}
     \caption{{\small Non-homeomorphic covers of the singular subspace, a $1$-dimensional orbicomplex with ramification points of order two. The three graph covering maps can be realized by rotation by $\pi$ about a center point in the embedding of the graph in the plane, and the orbifold covering map can be realized by reflection about a vertical line in the $\Theta$-graph. }}
     \label{figure:singular}
    \end{figure}
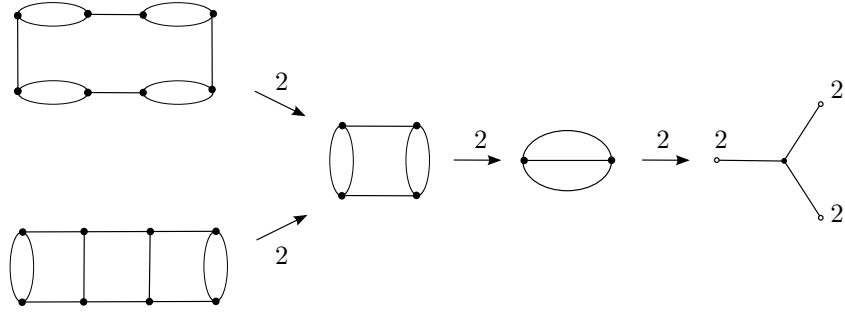

\begin{proof}[Proof of Theorem \ref{toprigidity}]
 
Let $\cD \in \mathcal{X}$ be the Davis orbicomplex
defined in Construction \ref{con:D} and
illustrated in Figure \ref{figure:orbicomplex}. To prove the class of spaces $\mathcal{X}$ is not topologically rigid, we will exhibit two covers of $\cD$ with the same fundamental group that are not homeomorphic. 

The orbicomplex $\cD$ has a singular subspace with underlying space a tripod formed by gluing together the non-reflection edges of the orbifolds in $\cP$; to prove the finite covers constructed are not homeomorphic, we prove the covers restricted to the singular subspace are not homeomorphic. The singular subspace of $\cD$ is a $1$-dimensional orbicomplex with underlying space a star on three vertices and so that each vertex of valence one is a ramification point of order two. The singular subspace and the covers restricted to the singular subspace are drawn in Figure \ref{figure:singular}. We prove that these graph coverings can be extended to finite coverings of the Davis orbicomplex $\cD$.

To construct two covers of $\cD$ that are not homeomorphic, we first construct covers of degree two, $\cX_2 \xrightarrow{2} \cX_1 \xrightarrow{2} \cD$, and then two further covers of degree two, $\cY \xrightarrow{2} \cX_2$ and $\cZ \xrightarrow{2} \cX_2$ so that $\cY$ and $\cZ$ are not homeomorphic but are homotopy equivalent. 

We begin by describing the cover $\cX_1 \xrightarrow{2} \cD$. This cover is illustrated in Figure \ref{figure:X1overD}. The space $\cX_1$ consists of two copies of $D^2(2,2,2,2,2,2)$, labeled $D_1$ and $D_2$, and four copies of $D^2(2,2,2,2)$, labeled $D_3, D_4, D_5, D_6$. To define the identification of these orbifolds, label the boundary circle of $D_i$ by first subdividing it into two segments of equal length by adding vertices $x_i$ and $y_i$. Label one oriented edge $\{x_i, y_i\}$ by $d_i$ and the other by $d_i'$. To form $\cX_1$, identify all vertices in $\{x_i \, | \, 1 \leq i \leq 6\}$ to form a vertex $x$ and identify all vertices in $\{y_i \, | \, 1 \leq i \leq 6\}$ to form a vertex~$y$. Identify edges $\{d_1, d_2, d_3, d_4\}$ to form a single edge $c_1$; identify edges $\{d_1',d_2', d_5',d_6'\}$ to form a single edge $c_2$; and, identify edges $\{d_3', d_4', d_5, d_6\}$ to form a single edge $c_3$. Then, for each $i$, the orbifold $D_i \subset \cX_1$ covers an orbifold $P_{\beta} \subset \cD$, so that if the boundary of $D_i$ is labeled $c_jc_k^{-1}$, then this curve double-covers the non-reflection edge of $P_{\beta}$ labeled $e_j^{-1}e_k$. These covering maps agree along the intersection of the set $\{D_i \, | \, 1 \leq i \leq 6\}$ in $\cX_1$, and hence the union of these spaces covers $\cD$, the union of the $P_{\beta}$, by degree two.

    \begin{figure}
     \begin{overpic}[scale=.8, tics=5]{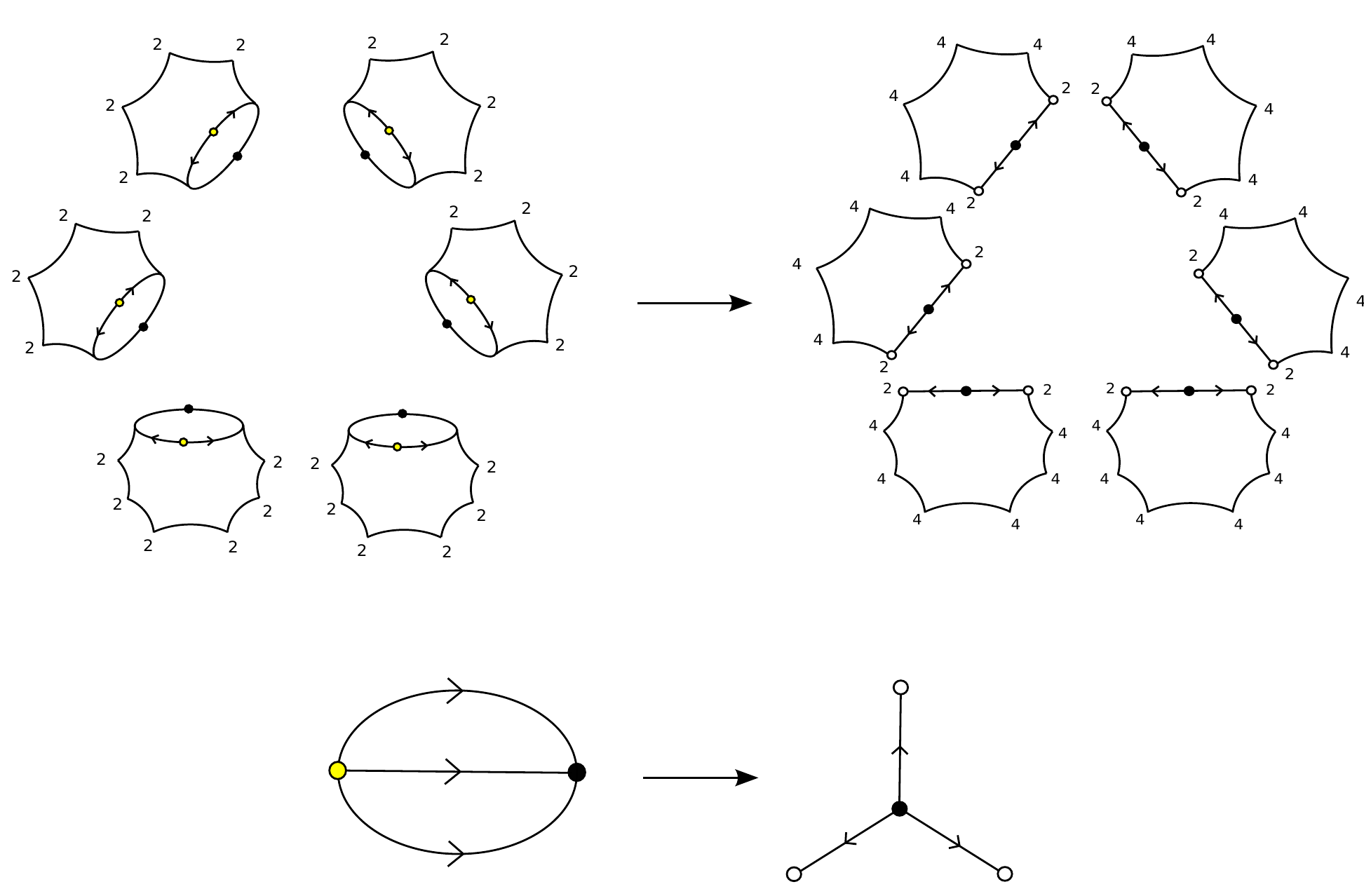}
     \put(50,45){\small{2}}
     \put(50,10){\small{2}}
     \put(55,0){\small{2}}
     \put(75,0){\small{2}}
     \put(65,17){\small{2}}
     \put(0,60){$\cX_1$}
     \put(96,60){$\cD$}
     \put(12,54){\small{$c_1$}}
     \put(5,42){\small{$c_1$}}
     \put(11,31.5){\small{$c_1$}}
     \put(26,31){\small{$c_1$}}
     \put(15,58){\small{$c_3$}}
     \put(7.5,45.5){\small{$c_3$}}
     \put(28,58){\small{$c_3$}}
     \put(34,45.5){\small{$c_3$}}
     \put(31,55){\small{$c_2$}}
     \put(37,42){\small{$c_2$}}
     \put(15,31.5){\small{$c_2$}}
     \put(31,31.3){\small{$c_2$}}
     \put(32,16.5){\small{$c_1$}}
     \put(32,10.8){\small{$c_2$}}
     \put(32,5){\small{$c_3$}}
     \put(67,11){\small{$e_1$}}
     \put(70,5.5){\small{$e_2$}}
     \put(60.5,5.5){\small{$e_3$}}
     \put(70.5,54){\small{$e_1$}}
     \put(64.5,42){\small{$e_1$}}
     \put(68,35){\small{$e_1$}}
     \put(84,35){\small{$e_1$}}
     \put(72,35){\small{$e_2$}}
     \put(88,35){\small{$e_2$}}
     \put(92.5,41.5){\small{$e_2$}}
     \put(86,54){\small{$e_2$}}
     \put(83,57){\small{$e_3$}}
     \put(90,44){\small{$e_3$}}
     \put(67,45){\small{$e_3$}}
     \put(73,57){\small{$e_3$}}
    \end{overpic}
     \caption{{\small Illustrated on the top row is the degree-$2$ cover $\cX_1 \rightarrow \cD$. The space $\cX_1$ contains six orbifolds, each with underlying space a disk and with either four or  six cone points of order two. The orbifolds are glued along the boundary of the disks as illustrated with the labeled arrows; all yellow vertices are identified, and all black vertices are identified. The cover restricted to the singular subspaces is illustrated below; the covering map is given by reflection through a vertical line through the left-hand graph.}}
     \label{figure:X1overD}
    \end{figure}

    \begin{figure}
     \begin{overpic}[scale=.33, tics=5]{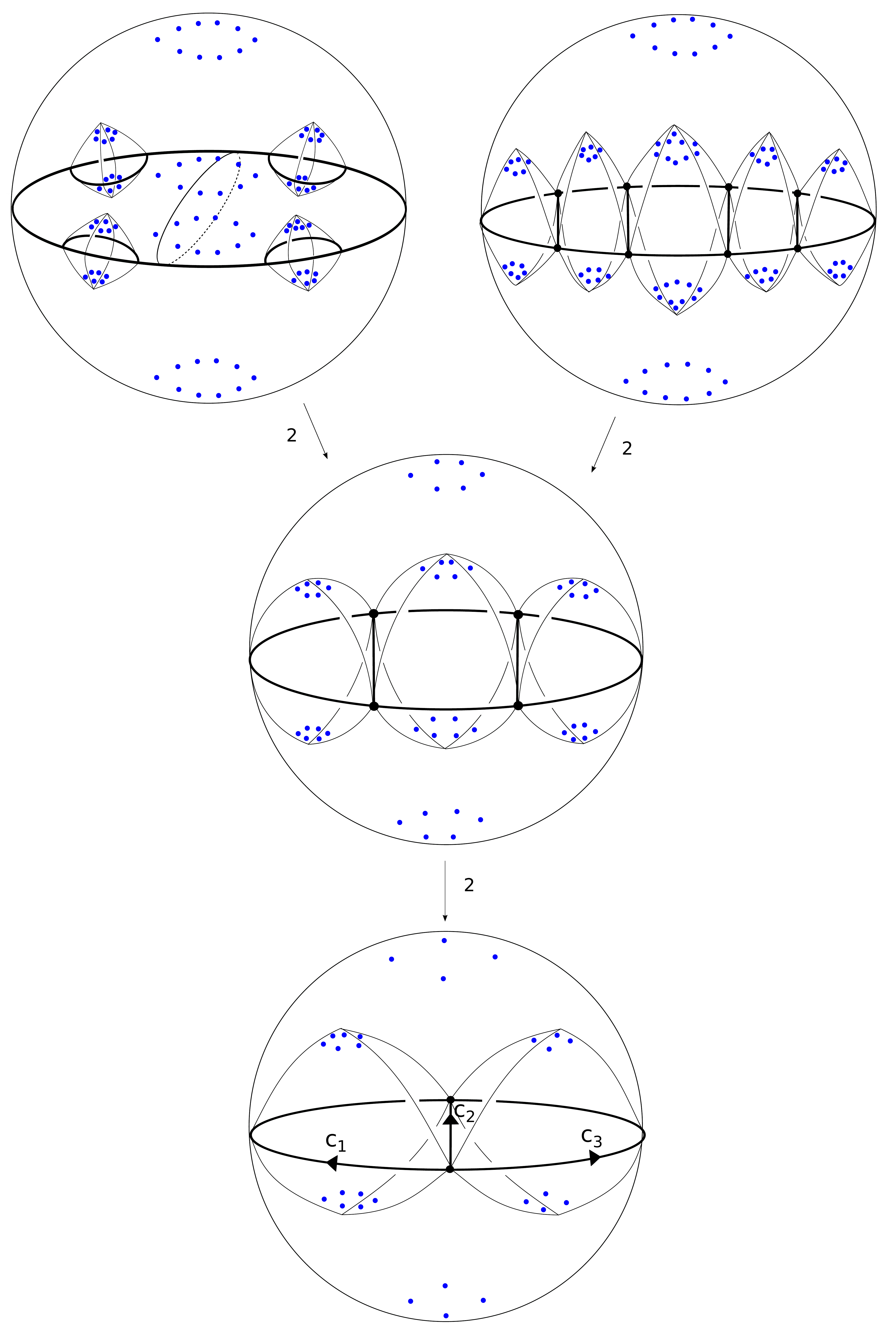}
     \put(12,17){$\cX_1$}
     \put(12,52){$\cX_2$}
     \put(0,95){$\cY$}
     \put(64,95){$\cZ$}
    \end{overpic}
\caption{{\Small Two coverings of an orbicomplex that are not homeomorphic, but which have the same orbifold fundamental group. The covering maps are each given by a $\Z/2\Z$ action of rotation about the $z$-axis. All of the blue points are orbifold points of order $2$.  }}
\label{r3_spheres}
\end{figure}

The remaining covering maps may be realized by rotation in $\R^3$; these are illustrated in Figure \ref{r3_spheres}. First, observe that the singular subspace of $\cX_1$ is the $\Theta$-graph, with two vertices of valence three and the three directed edges $\{c_1, c_2, c_3\}$ connecting the two vertices. This graph embeds in the plane and the boundary curves of the disk orbifolds $D_i$ are the three curves $c_1c_2^{-1}$, $c_1c_3^{-1}$, and $c_2c_3^{-1}$. Then, the space $\cX_1$ embeds in the $3$-ball $B^3 \subset \R^3$ as illustrated in Figure \ref{r3_spheres} so that the $\Theta$-graph embeds in the equatorial $xy$-plane. The copies of $D^2(2,2,2,2)$ with boundary $c_1c_3^{-1}$ may be viewed as the two hemispheres of the unit sphere. The copies of $D^2(2,2,2,2)$ with boundary $c_2c_3^{-1}$ may be viewed as the two hemispheres of a sphere embedded inside the unit sphere; likewise for the copies of $D^2(2,2,2,2,2,2)$. The remaining covering maps may be realized by rotations by $\pi$ about the $z$-axis in Euclidean space. The covering map restricted to each copy of $D^2(2,\ldots, 2)$ is either exactly $2$-to-$1$ or is the rotational covering map given in Lemma \ref{lemma:2fold}.

The two covers given in Figure \ref{r3_spheres} are not homeomorphic since their singular subspaces are not homeomorphic. It remains to show that these spaces are homotopy equivalent and hence have the same orbifold fundamental group. To see this, take a regular neighborhood of the singular locus of $\cY$ and $\cZ$ in the $xy$-plane to form a surface with genus zero and with six boundary components. The homotopy from the embedded graph to its regular neighborhood in the plane extends to homotopies from $\cY$ to a space $\cY'$ and from $\cZ$ to a space $\cZ'$, where the homotopy is the identity on the complement of the singular subspace. That is, the space $\cY'$ contains a surface of genus zero and with six boundary components $B_1, \ldots, B_6$. For $1 \leq i \leq 4$, two copies of $D^2(2,2,2,2,2,2)$ are identified to $B_i$ by homeomorphisms of the boundary curve of the disk orbifolds; for $5 \leq i \leq 6$, two copies of $D^2(2,2,2,2,2,2,2,2,2,2)$ are identified to $B_i$ by a homeomorphisms of the boundary curve of the disk orbifolds. The space $\cZ'$ is homeomorphic to $\cY'$. Thus, the orbicomplexes $\cY$ and $\cZ$ are homotopy equivalent, so their orbifold fundamental groups are isomorphic. \end{proof}

\subsection{Torsion-free covers}

   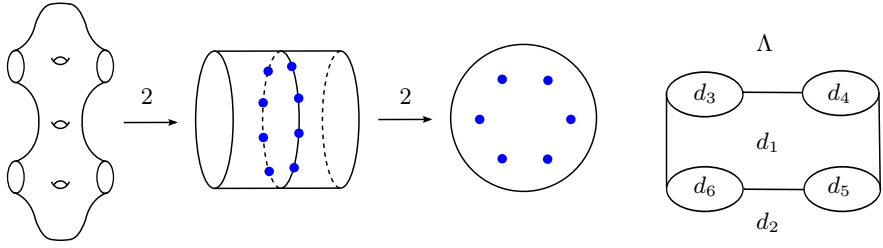
\begin{figure}
     \begin{overpic}[scale=.7, tics=5]{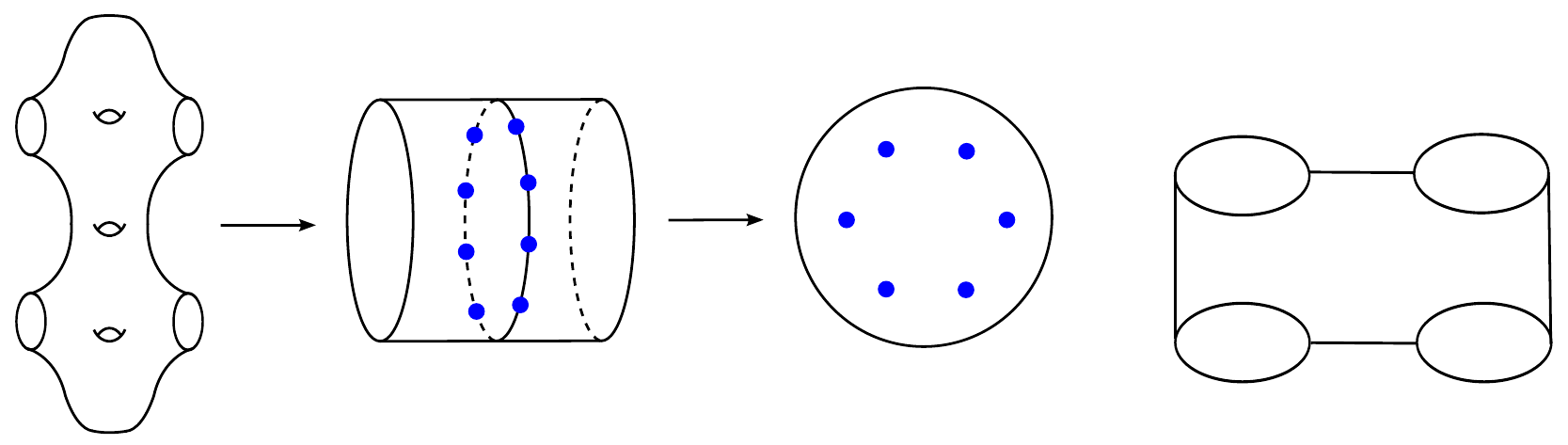}
     \put(16,16){\Small{$2$}}
     \put(45,16){\Small{$2$}}
     \put(85,22){\Small{$\Lambda$}}
     \put(85,11){\Small{$d_1$}}
     \put(85,2){\Small{$d_2$}}
     \put(78,16.5){\Small{$d_3$}}
     \put(93,16.5){\Small{$d_4$}}
     \put(93,6){\Small{$d_5$}}
     \put(78,6){\Small{$d_6$}}
     \end{overpic}     
\caption{{\Small On the left are degree-$2$ covers given by rotation by $\pi$ about a vertical axis positioned through the surface or orbifold. The blue points represent cone points of order $2$. On the right is the singular subspace of the orbicomplex $\cY$. }}
\label{figure:tf_covers}
\end{figure}

\begin{prop} \label{prop:tf_covers}
 Let $\cY$ and $\cZ$ be the orbicomplexes described in the proof of Theorem~\ref{toprigidity} and shown in Figure~\ref{r3_spheres}. There exist finite covers $\hY \rightarrow \cY$ and $\hZ \rightarrow \cZ$ so that $\pi_1(\hY)$ and $\pi_1(\hZ)$ are torsion-free, $\pi_1(\hY) \cong \pi_1(\hZ)$, and $\hY$ and $\hZ$ are not homeomorphic. 
\end{prop}
\begin{proof}
  We first describe the finite cover $\hY \rightarrow \cY$. The orbicomplex $\cY$ has a singular subspace the planar graph $\Lambda$, shown in Figure~\ref{figure:tf_covers}. Let $d_1 \ldots, d_6$ denote the boundary curves of the six planar regions in the complement of $\Lambda \subset \R^2$ as marked in Figure~\ref{figure:tf_covers}. Glued to $d_1$ and $d_2$ are two copies of $D^2(2,2,2,2,2,2,2,2,2,2)$, and glued to each of $d_3, \ldots, d_6$ are two copies of $D^2(2,2,2,2,2,2)$. As shown in Figure~\ref{figure:tf_covers}, the surface $S_{g,4}$ of genus $g$ and four boundary components forms a degree-$4$ cover of $D^2(\,\underbrace{2, \ldots, 2}_\text{$n$}\,)$, where $g=3$ if $n=6$ and $g=7$ if $n=10$. Indeed, embed the surface in $\R^3$ so that the boundary components are arranged symmetrically in pairs about a vertical axis and the holes of the surface lie along the axis. Rotate by $\pi$ about the vertical axis to produce an orbifold with underlying space an annulus and with $2g+2$ cone points of order $2$. Arrange the cone points symmetrically along a core curve of the annulus and rotate by $\pi$ about an axis that skewers the core curve in two non-singular points to obtain $D^2(\,\underbrace{2, \ldots, 2}_\text{$n$}\,)$. To form $\hY$, take four copies of the graph $\Lambda$; for $i=1,2$, glue to the four copies of $d_i$ the four boundary curves of $S_{7,4}$ by homeomorphisms; and, for $i=3, \ldots, 6$, glue to the four copies of $d_i$ the boundary curves of $S_{3,4}$ by homeomorphisms. Since each boundary curve of $S_{g,4}$ covers the boundary of $D^2(\,\underbrace{2, \ldots, 2}_\text{$n$}\,)$ by degree-$1$, $\hY$ forms a degree-$4$ cover of $\cY$. 
  
  The finite cover $\hZ \rightarrow \cZ$ is constructed similarly. By analogous arguments to those in the proof of Theorem~\ref{toprigidity}, $\hY$ and $\hZ$ are homotopy equivalent. The singular subspace of $\hY$ is homeomorphic to four copies of the singular subspace of $\cY$; likewise, the singular subspace of $\hZ$ is homeomorphic to four copies of the singular subspace of $\cZ$. Thus, $\hY$ and $\hZ$ are not homeomorphic. 
\end{proof}

\bibliographystyle{alpha}
\bibliography{refs}

\end{document}